\documentclass[12pt]{article}
\usepackage{amsmath,amssymb,amsthm,amsfonts,latexsym,amscd}
\usepackage{graphicx}
\newtheorem{theorem}{Theorem}[section]
\newtheorem{corollary}[theorem]{Corollary}
\newtheorem{remark}[theorem]{Remark}
\newtheorem{lemma}[theorem]{Lemma}
\newtheorem{proposition}[theorem]{Proposition}

\theoremstyle{remark}
\def\PG{\mbox{\rm PG}}

\begin{document}
\title{Characterizing parabolic hyperplanes of the hyperbolic and elliptic quadrics in $\PG(2n+1,q)$}
\author{Bikramaditya Sahu}
\date{}
\maketitle

\begin{abstract}
In this article, a combinatorial characterization of the family of parabolic hyperplanes of a hyperbolic (respectively, elliptic) quadric of $\PG(2n+1,q)$, using their intersection properties with the points and subspaces of codimension 2, is given.  \\


\end{abstract}

\section{Introduction}
%Let $\PG(n,q)$ denote the $n$-dimensional Desarguesian projective space defined over a finite field of order $q$, where $q$ is a prime power. The non-singular quadrics of $\PG(n,q)$ are very interesting objects with many combinatorial properties. One of the important properties of quadrics in $\PG(n,q)$ is that subspaces can only meet a quadric in certain ways. Therefore, we may form families of subspaces that all meet a particular quadric in the same way and then we can give a characterization of that family. 

The characterization of classical polar spaces by means of intersection properties with the subspaces of the original projective space can be traced back to the 1950's seminal work by Segre and Talini and their schools. This was continued by Beutelspacher, Thas and their collaborators in the 1970s and 1980s. One can refer to \cite{De-Du} for a detailed study in the past two decades. In 1991, Tallini and Ferri gave a characterization of the parabolic quadric of $\PG(4,q)$ by its intersection properties with planes and hyperplanes \cite{F-T}. Characterizations of non-singular quadrics and non-singular Hermitian varieties  by intersection numbers with hyperplanes and spaces of codimension 2 is given in \cite{De-Sch}. This generalizes the result of Tallini and Ferri.

A characterization of the family of planes meeting a non-degenerate quadric in $\PG(4,q)$ is given in \cite{But}. In a series of two recent papers \cite{BHJ,BHJS}, characterizations of elliptic hyperplanes and hyperbolic hyperplanes in $\PG(4,q)$ are given. In a recent paper \cite{Sc-Van} , the characterization of elliptic and hyperplanes in $\PG(2n,q)$ with respect to a parabolic quadric is given. In \cite{sahu}, a characterization of planes of $\PG(3,q)$ meeting a hyperbolic quadric in an irreducible conic is given. In this paper, we generalize the result in \cite{sahu} to arbitrary dimension by characterizing parabolic hyperplanes of a hyperbolic as well as an elliptic quadric in $\PG(2n+1,q)$. 

\begin{remark}
Throughout the paper we treat the similar elliptic and hyperbolic case simultaneously. We follow the following convention: we use the $\pm$ and the $\mp$ symbol, where $\pm$ reads as $+$ when we are in the hyperbolic case and $-$ in the elliptic case and vice versa for $\mp$. All statements and their proofs can be read by choosing either top or bottom row of the symbols $\pm$ and $\mp$ consistently. So, for every statement containing $\pm$ or $\mp$, there are two different statements.
\end{remark}
In this paper, we prove the following characterization theorem.

\begin{theorem} \label{main}
Let $\Sigma^\pm$ be a nonempty family of hyperplanes of $\PG(2n+1,q)$ for which the following properties are satisfied:
\begin{enumerate}
\item [(P1)] Every point of $\PG(2n+1,q)$ is contained in $q^n(q^n\mp 1)$ or $q^{2n}$ hyperplanes of $\Sigma^\pm$.
\item [(P2)] Every subspace of $\PG(2n+1,q)$ of co-dimension 2 which is contained in a hyperplane of $\Sigma^\pm$ is contained in at least $q-1$ hyperplanes of $\Sigma^\pm$.
\end{enumerate}
Then the following holds.
\begin{enumerate}
\item[(a)] For $n\geq 2$ and $q>2$, the collection $\Sigma^\pm$ is the set of all hyperplanes of $\PG(2n+1,q)$ meeting a non-singular quadric $\mathcal{Q}^\pm(2n+1,q)$ in a parabolic quadric.
\item[(b)] For $n=1$ and all $q$, the collection $\Sigma^+$ in $\PG(3,q)$ is the set of planes meeting a hyperbolic quadric in an irreducible conic.
\item[(c)] For $n=1$ and all $q$, the collection $\Sigma^-$ is the set of planes meeting an ovoid or $\Sigma^-$ is the set of planes meeting a fixed line in a unique point in $\PG(3,q)$. 
\end{enumerate}
\end{theorem}

\begin{remark}
(1) Consider the first case of Theorem \ref{main}(c) for the set $\Sigma^-$.  For $q$ odd, by Barloti \cite{Bar}, every ovoid in $\PG(3,q)$ is an elliptic quadric. So the collection $\Sigma^-$ will be the set of planes of $\PG(3,q)$ meeting an elliptic quadric in an irreducible conic. For $q$ even, there are non-classical ovoids in $\PG(3,q)$, so the planes meeting the ovoid in an oval satisfy properties (P1) and (P2).\\
(2) By a similar argument as that of \cite[Remark 1.3]{Sc-Van}, properties (P1) and (P2) of Theorem \ref{main} are independent of each other.
\end{remark}

We use the following two results \cite{De-Sch,sahu} to prove Theorem \ref{main}.

\begin{proposition}\cite[Theorem 1.6]{De-Sch}\label{de-sci}
If a point set $\mathcal{B}$ in $\PG(k,q)$, $k\geq 4$, $q>2$, or $\PG(4,2)$ has the same intersection numbers with respect to hyperplanes and subspaces of codimension 2 as a non-singular polar space $\mathcal{P}\in \{\mathcal{H}(k,q), \mathcal{Q}^+(k,q), \mathcal{Q}^-(k,q), \mathcal{Q}(k,q)\}$, then $\mathcal{B}$ is the point set of the non-singular polar space $\mathcal{P}$.
\end{proposition}

\begin{proposition}\cite[Theorem 1.1]{sahu}\label{sahu}
Let $\Sigma^+$ be a nonempty family of planes of $\PG(3,q)$, for which the following properties are satisfied:
\begin{enumerate}
\item [(P1)] Every point of $\PG(3,q)$ is contained in  $q^2-q$ or $q^2$ planes of $\Sigma^+$.
\item [(P2)] Every line of $\PG(3,q)$ is contained in $0$, $q-1$, $q$ or $q+1$ planes of $\Sigma^+$.
\end{enumerate}
Then $\Sigma^+$ is the set of all planes of $\PG(3,q)$ meeting a hyperbolic quadric in an irreducible conic.
\end{proposition}

Note that Theorem \ref{main}(b) directly follows from Proposition \ref{sahu}. Proof of Theorem \ref{main}(b) and (c) can be seen directly. Here we give a proof which works for all the cases.

A {\it quasi-quadric} in a projective space $\PG(k,q)$ is a set of points that has the same intersection numbers with respect to hyperplanes as a non-degenerate quadric in that space. It is clear that, non-degenerate quadrics are examples of quasi-quadrics, but many other examples also exist. We refer to \cite{DHOP} for a detailed study on quasi-quadrics. 
  
We first proof the following 
\begin{proposition} \label{main-prop}
Let $\Sigma^\pm$ be a non empty family of hyperplanes of $\PG(2n+1,q)$ such that
\begin{enumerate}
\item [(P1)] Every point of $\PG(2n+1,q)$ contained in  $q^n(q^n\mp 1)$ or $q^{2n}$ hyperplanes of $\Sigma^\pm$.

\end{enumerate}
Then the set $B^\pm$ of points contained in $q^n(q^n\mp 1)$ hyperplanes of $\Sigma^\pm$ forms a quasi-quadric and $\Sigma^\pm$ is the set of all hyperplanes of $\PG(2n+1,q)$ that meets $B^\pm$ in $|\mathcal{Q}(2n,q)|$ points.
\end{proposition}

\section{Preliminaries}
\subsection{On quadrics in $PG(k,q)$}
These are all the finite non-singular polar spaces;
\begin{enumerate}
\item[$\bullet$] a hyperbolic quadric in $\PG(2n+1,q)$, denoted by $\mathcal{Q}^+(2n+1,q)$;
\item[$\bullet$]  a parabolic quadric in $\PG(2n,q)$, denoted by $\mathcal{Q}(2n,q)$;
\item[$\bullet$]  an elliptic quadric in $\PG(2n+1,q)$, denoted by $\mathcal{Q}^-(2n+1,q)$;
\item[$\bullet$]  a Hermitian quadric in $\PG(n,q^2)$, denoted by $\mathcal{H}(n,q^2)$;
\item[$\bullet$]  a sympletic polar space in $\PG(2n+1,q)$, denoted by $\mathcal{W}_{2n+1}(q)$.
\end{enumerate}

We recall the following properties of non-singular quadrics of $\PG(2n+1,q)$, one can refer to \cite{Hir-2} for the details.  The number of points in $\mathcal{Q}^\pm(2n+1,q)$ is $\dfrac{(q^{n+1}\mp 1)(q^n\pm 1)}{q-1}$.  A hyperplane of $\PG(2n+1,q)$ can intersect a $\mathcal{Q}^\pm(2n+1,q)$ either in a $\mathcal{Q}(2n,q)$ or in a cone $p\mathcal{Q}^\pm(2n-1,q)$. Through a point of $\mathcal{Q}^\pm(2n+1,q)$ there is a unique tangent hyperplane meeting $\mathcal{Q}^\pm(2n+1,q)$ in a cone $p\mathcal{Q}^\pm (2n-1,q)$. So, there are $\dfrac{q^{2n+2}-1}{q-1}-\dfrac{(q^{n+1}\mp 1)(q^n\pm 1)}{q-1}=q^n(q^{n+1}\mp 1)$ parabolic hyperplanes in $\PG(2n+1,q)$ with respect to $\mathcal{Q}^\pm(2n+1,q)$. A subspace of codimension 2 can intersect a $\mathcal{Q}^\pm (2n+1,q)$ in a $\mathcal{Q}^+(2n-1,q)$, in a cone $p\mathcal{Q}(2n-2,q)$, in a $\mathcal{Q}^-(2n-1,q)$ or in a cone $L\mathcal{Q}^\pm (2n-3,q)$.

So the intersection numbers with hyperplanes are 

$$h_1^\pm=\dfrac{q^{2n}-1}{q-1},\ \ \ h_2^\pm=1+q\dfrac{(q^n\mp 1)(q^{n-1}\pm 1)}{q-1}.$$
The intersection numbers with subspaces of codimension 2 are 
$$c_1^\pm=\dfrac{(q^n\pm 1)(q^{n-1}\mp 1)}{q-1},\ \ \  c_2^\pm=1+\dfrac{q(q^{2n-2}-1)}{q-1},$$
$$c_3^\pm=\dfrac{(q^n\mp 1)(q^{n-1}\pm 1)}{q-1},\ \ \  c_4^\pm=1+q+\dfrac{q^2(q^{n-1}\mp 1)(q^{n-2}\pm 1)}{q-1}.$$
For the above numbers one can also see \cite[Section 2.2]{De-Sch}.

\begin{proposition}\label{prop-verify}
Let $\Sigma^\pm$ be the set of hyperplanes meeting a non-singular quadric $\mathcal{Q}^\pm(2n+1,q)$ in $|\mathcal{Q}(2n,q)|$ points. Then the collection $\Sigma^\pm$ satisfies Conditions (P1) and (P2) of Theorem \ref{main}. 
\end{proposition}

\begin{proof}
(P1) Let $x$ be a point of $\PG(2n+1,q)$ and  let $m^\pm$ be the number of hyperplanes of $\Sigma^\pm$ through $x$. Assume that $x\in \mathcal{Q}^\pm(2n+1,q)$. We count in two different ways the number of order pairs of point-hyperplanes 

$$\{(x,\pi)\ |\ x\in \mathcal{Q}^\pm (2n+1,q), \pi \in \Sigma^\pm\ \mbox{ and } x\in \pi\}.$$ Thus we have 
$$|\mathcal{Q}^\pm(2n+1,q)|m^\pm=q^n(q^{n+1}\mp 1)\dfrac{q^{2n}-1}{q-1},$$ this gives 
$m^\pm=q^n(q^n\mp 1).$ 

Now assume that $x$ is not a point of $\mathcal{Q}^\pm (2n+1,q)$. We count in two different ways the number of order pairs of point-hyperplanes 

$$\{(x,\pi)\ |\ x\in \PG(2n+1,q)\setminus \mathcal{Q}^\pm (2n+1,q), \pi \in \Sigma^\pm\ \mbox{ and } x\in \pi\}.$$ In this case we have
$$\left(|\PG(2n+1,q)|-|\mathcal{Q}^\pm(2n+1,q)|\right)m^\pm=q^n(q^{n+1}\mp 1)\left(\dfrac{q^{2n+1}-1}{q-1}-\dfrac{q^{2n}-1}{q-1}\right),$$ this gives 
$m^\pm=q^{2n}.$ This verifies (P1).

(P2) There are four types of codimension 2 subspaces with respect to $\mathcal{Q}^\pm(2n+1,q)$. Let $\Pi$ be a codimension 2 subspace intersecting $\mathcal{Q}^\pm(2n+1,q)$ in $c_i^\pm$ points and let $s_i^\pm$ be the number of hyperplanes of $\Sigma^\pm$ containing $\Pi$ for $i=1,2,3$ and $4$. Note that the total number of points of $\mathcal{Q}^\pm(2n+1,q)$ in $\PG(2n+1,q)$ can be obtained by taking all the hyperplanes containing $\Pi$ and then counting points of $\mathcal{Q}^\pm(2n+1,q)$ in each such hyperplanes. So, we obtain the following

$$s_i^\pm(|\mathcal{Q}(2n,q)|-c_i^\pm)+(q+1-s_i^\pm)(|p\mathcal{Q}^\pm(2n-1)|-c_i^\pm)=|\mathcal{Q}^\pm(2n+1)|-c_i^\pm.$$ Solving the equation in $s_i^\pm$ for $i=1,2,3,4$, we get that a codimension 2 subspace of $\PG(2n+1,q)$ is contained in $0$, $q-1, q$ or $q+1$ hyperplanes of $\Sigma^\pm$. 
\end{proof}

\subsection{Proof of Proposition \ref{main-prop}}

Let $\Sigma^\pm$ be a nonempty set of hyperplanes of $\PG(2n+1,q)$ for which the property (P1) stated in Proposition \ref{main-prop} holds. A point of $\PG(2n+1,q)$ is said to be {\it black} or {\it white} according as it contained in  $q^n(q^n\mp 1)$ or $q^{2n}$ hyperplanes of $\Sigma^\pm$. Let $B^\pm$ be the set of black points in $\PG(2n+1,q)$ with respect to $\Sigma^\pm$. Let $b^\pm$ and $w^\pm$, respectively, denote the number of black points and the number of white points in $\PG(2n+1,q)$ with respect to $\Sigma^\pm$. We have 
\begin{equation}\label{eq-1}
w^\pm=\frac{q^{2n+2}-1}{q-1}-b^\pm.
\end{equation}

Counting in two ways the point-hyperplane incident pairs,
$$\{(x,\pi)\ | \ x \ \mbox{is a point in $\PG(2n+1,q)$}, \ \pi \in \Sigma^\pm \ \mbox{and}\ x\in \pi \},$$
we obtain  
\begin{equation}\label{eq-2}
b^\pm q^n(q^n\mp 1)+w^\pm q^{2n}=|\Sigma^\pm|\left(\frac{q^{2n+1}-1}{q-1}\right).
\end{equation}

From Equations (\ref{eq-1}) and (\ref{eq-2}), eliminating $w^\pm$, we obtain 
\begin{equation}\label{eq-2b}
q^nb^\pm=\mp \left(\frac{q^{2n+1}-1}{q-1}\right)|\Sigma^\pm|\pm q^{2n}\left(\frac{q^{2n+2}-1}{q-1}\right).
\end{equation}
From Equation (\ref{eq-2b}), we have the following

\begin{lemma}\label{lem-divide}
$|\Sigma^\pm|=q^nr^\pm$ for some $1\leq r^\pm \leq \frac{q^{2n+2}-1}{q-1}$
\end{lemma} 
\begin{proof}
Observe that $q^n$ and $\dfrac{q^{2n+1}-1}{q-1}$ are coprime. So from Equation (\ref{eq-2b}), it can been seen that $q^n$ divides both the terms on the right hand side and hence has to divide $|\Sigma^\pm|$.
\end{proof}

From now on by Lemma \ref{lem-divide}, {\it $r^\pm$ is an integer such that $|\Sigma^\pm|=q^nr^\pm$ and $1\leq r^\pm \leq \frac{q^{2n+2}-1}{q-1}$}.

\begin{lemma}\label{lem-black-fixed}
The number of black points in a hyperplane of $\Sigma^\pm$ is $\left(\frac{q^{2n}-1}{q-1}\right)(\pm q^{n+1}\mp r^\pm)$.
\end{lemma}

\begin{proof}
Let $\pi$ be a hyperplane in $\Sigma^\pm$. Let $b_\pi^\pm$ and $w_\pi^\pm$, respectively, denote the number of black and white points in $\pi$ with respect to $\Sigma^\pm$. Then $w_\pi^\pm=\dfrac{q^{2n+1}-1}{q-1}-b_\pi^\pm$.

By counting the incident point-hyperplane pairs of the following set,
$$\{(x,\sigma)\ | \ x \ \mbox{is a point in $\PG(2n+1,q)$}, \ \sigma \in \Sigma^\pm \ , \sigma\neq \pi \ \mbox{and} \ x\in \pi\cap \sigma \},$$ we get\\ 
$$b_\pi^\pm(q^{2n}+q^n-1)+w_\pi^\pm (q^{2n}-1)=(|\Sigma^\pm|-1)\left(\frac{q^{2n}-1}{q-1}\right).$$ 

Since $w_\pi^\pm=\dfrac{q^{2n+1}-1}{q-1}-b_\pi^\pm$ and $|\Sigma^\pm|=q^nr^\pm$, it follows from the above equation that 
$$b_{\pi}^\pm=\left(\frac{q^{2n}-1}{q-1}\right) (\pm q^{n+1}\mp r^\pm).$$
\end{proof}

Now, counting in two ways the incident triples,
$$\{(x,\pi,\sigma)\ | \ x \ \mbox{is a point in $\PG(2n+1,q)$}, \ \pi,\sigma \in \Sigma^\pm, \pi\neq \sigma \ \mbox{and}\ x\in \pi\cap\sigma \},$$
we get  
\begin{equation*}
b^\pm(q^{2n}\mp q^n)(q^{2n}\mp q^n-1)+w^\pm q^{2n}(q^{2n}-1)=|\Sigma^\pm|(|\Sigma^\pm|-1)\left(\frac{q^{2n}-1}{q-1}\right).
\end{equation*}

Putting the values of $w^\pm$ and $|\Sigma^\pm|$ respectively from Equation (\ref{eq-1}) and Lemma \ref{lem-divide} in the above equation, we obtain
\begin{equation}\label{eq-3}
b^\pm=\left(\frac{q^n\pm 1}{(q-1)(2q^n\pm 1)}\right)\left[ \pm q^n(q^{2n+2}-1)\mp r^\pm (q^nr^\pm-1)\right].
\end{equation}

\begin{corollary}\label{coro-bound-r} We have
\begin{enumerate}
\item[(i)]  $q^{n+1}-q\leq r^+\leq q^{n+1}- 1$.
\item[(ii)] $q^{n+1}+1\leq r^-\leq q^{n+1}+q$.
\end{enumerate}
\end{corollary}
\begin{proof}
Let $\pi$ be a hyperplane of $\Sigma^\pm$ and let $b_\pi^\pm=|\pi\cap B^\pm|$. Then by Lemma \ref{lem-black-fixed}, we have $0\leq b_\pi^\pm=\left( \dfrac{q^{2n}-1}{q-1}\right)(\pm q^{n+1}\mp r^\pm)\leq |\pi|=\dfrac{q^{2n+1}-1}{q-1}$. It follows that $q^{n+1}-q\leq r^+\leq q^{n+1}$, and $q^{n+1}\leq r^-\leq q^{n+1}+q$.

Suppose that $r^\pm=q^{n+1}$. From Equation (\ref{eq-3}), we get that

$$b^\pm=\dfrac{(q^n\pm 1)(\pm q^{n+1}\mp q^n)}{(q-1)(2q^n\pm 1)}.$$

This gives a contradiction as $b^\pm$ is a positive integer. This proves the corollary.
\end{proof}

We need the following result from basic number theory.
\begin{lemma}\label{lem-useful}
Let $k$ be such that $1\leq k\leq q$, and assume that $n\neq 1$ in the elliptic case. Then $q^n\pm 1$ divides $k(q^{2n+1}-1)\pm q^{n+1}\mp q^n$ if and only if $k=1$.
\end{lemma}

\begin{proof}
Assume that $n\neq 1$ in the elliptic case. Observe that for $k=1$, we have $q^n+1$ divides $k(q^{2n+1}-1)+q^{n+1}-q^n=(q^{n+1}-1)(q^n+1)$. Conversely, suppose that $q^n\pm 1$ divides $k(q^{2n+1}-1)\pm q^{n+1}\mp q^n$. We separately consider two cases.

Suppose that $q^n+1$ divides $k(q^{2n+1}-1)+q^{n+1}-q^n$. Observe that $k(q^{2n+1}-1)+q^{n+1}-q^n=(kq^{n+1}-1)(q^n+1)-(k-1)(q^{n+1}+1)$. Since $q^n+1$ and $q^{n+1}+1$ are co-prime, we have $q^n+1$ divides $k-1$ for $1\leq k\leq q$, this forces $k=1$. (In this case it is true for $n=1$ also)

Similarly, suppose that $q^n-1$ divides $k(q^{2n+1}-1)-q^{n+1}+q^n$. Observe here also that $k(q^{2n+1}-1)-q^{n+1}+q^n=(kq^{n+1}+1)(q^n-1)+(k-1)(q^{n+1}-1)$. So $q^n-1$ divides $(k-1)(q^{n+1}-1)$, that is, $q^{n-1}+\cdots+q+1$ divides $(k-1)(q^n+\cdots +q+1)$ for $1\leq k\leq q$ which is possible if and only if $k=1$.    
\end{proof}

%It can be observed in the proof of Lemma \ref{lem-useful} that the statement holds true for $n=1$ in the hyperbolic case. 

\begin{lemma}\label{r-value}
$r^\pm=q^{n+1}\mp 1$, where $n\neq 1$ in the elliptic case.
\end{lemma}

\begin{proof}
We know from Corollary \ref{coro-bound-r} that $r^\pm=q^{n+1}\mp k$ for some $1\leq k\leq q$. We show that $k=1$.

We know that there are $q^n(q^n\mp 1)$ hyperplanes containing a black point. By Lemma \ref{lem-black-fixed}, there are $\left(\frac{q^{2n}-1}{q-1}\right)(\pm q^{n+1}\mp r^\pm)$ black points contained in a hyperplane of $\Sigma^\pm$. Now, counting in two ways the point-hyperplane incident pairs,
$$\{(x,\pi)\ | \ x \ \mbox{is a black point in $\PG(2n+1,q)$}, \ \pi \in \Sigma^\pm \ \mbox{and}\ x\in \pi \},$$
we obtain  

$$b^\pm q^n(q^n\mp 1)=|\Sigma^\pm|\left( \frac{q^{2n}-1}{q-1} (\pm q^{n+1}\mp r^\pm) \right).$$ Putting $|\Sigma^\pm|=q^nr^\pm$ and solving for $b^\pm$, we get that 
$$b^\pm=r^\pm\left( \frac{q^n-1}{q-1}(\pm q^{n+1}\mp r^\pm)\right).$$ 

From Equation (\ref{eq-2b}) and Lemma \ref{lem-divide}, we get 
$$b^\pm=\pm q^{n}\left(\frac{q^{2n+2}-1}{q-1}\right)\mp \left(\frac{q^{2n+1}-1}{q-1}\right)r^\pm.$$

Equating the above two equations for $b^\pm$, we get 
$$r^\pm(q^n\pm 1)(\pm q^{n+1}\mp r^\pm)=\pm(q^{2n+2}-1)q^n\mp (q^{2n+1}-1)r^\pm.$$

Putting $r^\pm=q^{n+1}\mp k$ in the above equation, we obtain
$$(q^{n+1}\mp k)(q^n\pm 1)k=\pm (q^{2n+2}-1)q^n\mp (q^{2n+1}-1)(q^{n+1}\mp k)$$
$$=k(q^{2n+1}-1)\pm q^{n+1}\mp q^n.$$

The above equation tells that $q^n\pm 1$ divides $k(q^{2n+1}-1)\pm q^{n+1}\mp q^n$, which is possible by Lemma \ref{lem-useful} if and only if $k=1$. 
This proves the lemma.
\end{proof}

\begin{lemma}\label{valur-r-n=1}
Let $\Sigma^-$ be the collection of planes in $\PG(3,q)$ satisfying property (P1) of Proposition \ref{main-prop}. Then $r^-=q^2+q$ or $q^2+1$.
\end{lemma}

\begin{proof}
Let $\Sigma^-$ be the collection of planes satisfying property (P1) of Proposition \ref{main-prop}. Let $r^-=q^2+k$. In the proof of Lemma \ref{r-value}, we have $(q^2+k)(q-1)k=k(q^3-1)-q^2+1$, that is, $(q-k)(k-1)=0$. So, $k=1$ or $k=q$. If $k=q$, then $r^-=q^2+q$ and if $k=1$, then $r^-=q^2+1$
\end{proof}

From Lemma \ref{valur-r-n=1}, {\it for the rest of this section we assume that $r-=q^2+1$ in the elliptic case when $n=1$}, so that $r=q^{n+1}\mp 1$ for all $n\geq 1$. We separately consider the case $n=1$ and $r^-=q^2+q$ in the last part of this paper. The following follows from Lemmas \ref{lem-divide}, \ref{lem-black-fixed} and \ref{r-value}.

\begin{corollary}\label{sizes}
The following hold.
\begin{enumerate}
\item[(i)] $|\Sigma^\pm|=q^n(q^{n+1}\mp 1).$
\item[(ii)] The number of black points $b^\pm=\dfrac{(q^{n+1}\mp 1)(q^n\pm 1)}{q-1}$.
\item[(iii)] The number of black points $b_\pi^\pm$ in a hyperplane of $\pi$ of $\Sigma^\pm$ is $\dfrac{q^{2n}-1}{q-1}$.
\end{enumerate}
\end{corollary}

\begin{lemma}\label{black-tangent}
Any hyperplane of $\PG(2n+1,q)$ not in $\Sigma^\pm$ contains $\dfrac{q^{2n}\pm q^{n+1}\mp q^n-1}{q-1}$ black points. 
\end{lemma}

\begin{proof}
We prove the lemma with a similar argument as that of Lemma \ref{lem-black-fixed}. Let $\pi$ be a hyperplane of $\PG(3,q)$ not in $\Sigma^\pm$. Let $b_\pi^\pm$ and $w_\pi^\pm$ denote the number of black and white points, respectively, in $\pi$. Then $w_\pi^\pm=\dfrac{q^{2n+1}-1}{q-1}-b_\pi^\pm$.

By counting the incident point-hyperplane pairs of the following set,
$$\{(x,\sigma)\ | \ x \ \mbox{is a point in $\PG(2n+1,q)$}, \ \sigma \notin \Sigma^\pm \ , \sigma\neq \pi \ \mbox{and} \ x\in \pi\cap \sigma \},$$ we get\\ 
$$b_\pi^\pm\left(\dfrac{q^{2n+1}-1}{q-1}-(q^{2n}\mp q^n)-1\right)+ \left(\dfrac{q^{2n+1}-1}{q-1}-b_\pi^\pm\right) \left( \dfrac{q^{2n+1}-1}{q-1}-q^{2n}-1\right)$$ 
$$=\left( \dfrac{q^{2n+2}-1}{q-1}-|\Sigma^\pm|-1\right)\left( \dfrac{q^{2n}-1}{q-1}\right).$$

Since by Corollary \ref{sizes}, $|\Sigma^\pm|=q^n(q^{n+1}\mp 1)$, the above equation simplifies to give 
$$b_\pi^\pm= \dfrac{q^{2n}\pm q^{n+1}\mp q^n-1}{q-1}.$$ This proves the lemma.
\end{proof}

Corollary \ref{sizes}(iii) together with Lemma \ref{black-tangent} proves Proposition \ref{main-prop}.

\subsection{Proof of Theorem \ref{main}}
Let $\Sigma^\pm$ be a nonempty set of hyperplanes satisfying properties (P1) and (P2) of Theorem \ref{main}. Recall that $B^\pm$ is the set of black points in $\PG(2n+1,q)$ with respect to $\Sigma^\pm$.
 
\begin{lemma}\label{lem-p2}
The number of black points in a codimension $2$ subspace of $\PG(2n+1)$ is either $c_1^\pm,c_2^\pm,c_3^\pm$ or $c_4^\pm$.
\end{lemma}

\begin{proof}
Let $\Pi^\pm$ be a codimension 2 subspace of $\PG(2n+1)$ and let $m^\pm$ be the number of black points contained in $\Pi^\pm$. Let $s^\pm$ be the number of hyperplanes of $\Sigma^\pm$ containing $\Pi^\pm$. Note that the total number of points of $B^\pm$ in $\PG(2n+1,q)$ can be obtained by taking all the hyperplanes containing $\Pi^\pm$ and then counting points of $B^\pm$ in each such hyperplanes. We apply a similar argument as that of the proof of Proposition \ref{prop-verify}. So, for any hyperplane $\pi$ of $\Sigma^\pm$ and any hyperplane $\alpha$ not in $\Sigma^\pm$, applying Corollary \ref{sizes} (iii) and Lemma \ref{black-tangent}, we obtain the following

$$s^\pm(|\pi\cap B^\pm|-m^\pm)+(q+1-s^\pm)\left(|\alpha\cap B^\pm|-m^\pm\right)=b^\pm-m^\pm.$$ Since by property (P2) of Theorem \ref{main}, there are four choices for $s^\pm$, namely $0,q-1,q$ and $q+1$. Using Corollary \ref{sizes}, Lemma \ref{black-tangent} and solving the equation in $s^\pm$ for $s^\pm=0,q-1,q,q+1$, we obtain the number of black points in $\Pi^\pm$ is either $c_1^\pm,c_2^\pm,c_3^\pm$ or $c_4^\pm$ (see section 2.1).  
\end{proof}

{\it For the next result assume that $r^-=q^2+1$ in the elliptic case of $\PG(3,q)$}.

\begin{corollary}\label{main-a-b}
The set $\Sigma^\pm$ is the set of all hyperplanesplanes of $\PG(2n+1,q), q>2$,  $q$ is odd for $\Sigma^-$ in $\PG(3,q)$, meeting a non-singular quadric $\mathcal{Q}^\pm(2n+1)$ in a parabolic quadric. For $q$ is even and $n=1$, the collection $\Sigma^-$ is the set of planes meeting an ovoid in $\PG(3,q)$.
\end{corollary}
\begin{proof}

Let $n=1$. By Proposition \ref{sahu}, $\Sigma^+$ is the set of parabolic planes of $\PG(3,q)$ with respect to the hyperbolic quadric $B^+$. By Lemma \ref{lem-p2}, the set of black points $B^-$ forms an ovoid of $\PG(3,q)$. For $q$ odd, by Barloti \cite{Bar}, $B^-$ will be an elliptic quadric and  $\Sigma^-$ is the set of parabolic planes of $\PG(3,q)$ with respect to the elliptic quadric $B^-$.

For $n\geq 2$ and $q>2$, the corollary follows from Proposition \ref{main-prop} and, Proposition \ref{de-sci}. 
\end{proof}

%Corollary \ref{main-a-b} proves (a) and (b) of Theorem \ref{main}.

%\subsection{Proof of Theorem \ref{main}(c)}
Now {\it for the rest of the paper assume that $n=1$ and $r^-=q^2+q$}.

Let $\Sigma^-$ be the collection of planes in $\PG(3,q)$ satisfying properties (P1) and (P2) of Theorem \ref{main}. Recall that $B^-$ is the set of black points in $\PG(3,q)$. Then the following follows from section 2.2.

\begin{lemma}\label{sizes-n=1}
The following hold.
\begin{enumerate}
\item[(i)] $|\Sigma^-|=q^3+q^2.$
\item[(ii)] The number of black points $b^-=|B^-|=q^3+q^2$.
\item[(iii)] The number of black points $|\pi\cap B^-|$ in a hyperplane $\pi$ of $\PG(3,q)$ is $q^2+q$ or $q^2$ depending on whether $\pi\in \Sigma^-$ or not.  
\end{enumerate}
\end{lemma}

\begin{proof}
Since $r^-=q^2+q$, (i) follows from Lemma \ref{lem-divide} and hence (ii) follows from Equation (\ref{eq-2b}). Let $\pi$ be a plane of $\PG(3,q)$. If $\pi\in \Sigma^-$, then $|\pi\cap B^-|=q^2+q$, by Lemma \ref{lem-black-fixed}. If $\pi\notin \Sigma^-$, by a similar proof as that of Lemma \ref{black-tangent}, it follows that $|\pi\cap B^-|=q^2$.
\end{proof}

\noindent We use the following version of Bose and Burton \cite[Theorem 1]{BB} for $k=3$ in the proof of the next result.
\begin{proposition}\label{BB}
If $A$ is a set of points meeting all the lines of $\PG(2,q)$, then $|A|\geq q+1$, and equality holds if and only of $A$ is a line of $\PG(2,q)$.
\end{proposition}

\begin{theorem}\label{main-n=1}
The set of black points $B^-$ is the complement of a line $l$ in $\PG(3,q)$ and $\Sigma^-$ is the set planes of $\PG(3,q)$ meeting $l$ in a unique point. 
\end{theorem}

\begin{proof}
From Lemma \ref{sizes-n=1}(ii), the set of white point $W^-$ in $\PG(3,q)$ contains $q+1$ points. Again by Lemma \ref{sizes-n=1}(iii), any plane $\alpha$ not in $\Sigma^-$, it follows that $|\alpha\cap W^-|=w-=q+1$. We claim that $\alpha\cap W^-$ is a line. 

If there is a line $m$ of $\alpha$ disjoint from $\alpha\cap W^-$, then each of the $q$ planes through $m$, except $\alpha$, is a plane of $\Sigma^-$ in $\PG(3,q)$. Hence the total number of white points is greater than or equals to $q\cdot 1+q+1=2q+1$, which is a contradiction. So each line of $\alpha$ meets  $\alpha\cap W^-$. Hence by Proposition \ref{BB} , $\alpha\cap W^-$ is a line of $\PG(3,q)$. So all the $q+1$ planes not in $\Sigma^-$ are the planes through the line $\alpha\cap W^-$. The rest now follows. 
\end{proof}

Corollary \ref{main-a-b} together with Theorem \ref{main-n=1} proves Theorem \ref{main}.
\bigskip

\noindent {\bf Conclusion:} In this article, a classification of all the parabolic hyperplanes with respect to a hyperbolic and elliptic quadric in $\PG(2n+1,q)$ is given for all $n\geq 1$, $q> 2$. In addition, a classification of all the parabolic hyperplanes (planes meeting in an irreducible conic) with respect to a hyperbolic quadric in $\PG(3,q)$ is given for all $q$. Only for $n=1$ in the elliptic case, two more types of examples occur for $\Sigma^-$. For the case, $n\geq 3$ and $q=2$ is not explained in this paper.  
\bigskip

\noindent {\bf Acknowledgment:} I would like to thank Dr. Binod Kumar Sahoo for inviting me for the Annual Foundation School-II at NISER Bhubaneswar during June-July 2022 where some of the work has been done.

\vskip .5cm

\noindent {\bf Address}:\\

\noindent {\bf Bikramaditya Sahu} (Email: sahuba@nitrkl.ac.in)\\
 Department of Mathematics, National Institute of Technology\\
 Rourkela - 769008, Odisha, India.\\
\end{document}